\documentclass[12pt, reqno]{amsproc}
\usepackage[utf8]{inputenc}
\usepackage[T2A]{fontenc}
\usepackage[english]{babel}

\usepackage[margin = 2.2cm]{geometry}
\usepackage{enumitem}
\usepackage[svgcolors, dvipsnames]{xcolor}
\usepackage{graphicx}
\usepackage[all]{xy}
\usepackage{amssymb, amsmath, amscd, amsthm}
\usepackage[colorlinks=true]{hyperref}
\hypersetup{urlcolor=blue, citecolor=red}

\title[Morse-Bott foliation on the solid Klein bottle]{Diffeomorphism groups of Morse-Bott foliation on the solid Klein bottle by Klein bottles parallel to the boundary}

\author{Sergiy Maksymenko}
\address{Algebra and Topology Department, Institute of Mathematics of National Academy of Sciences of Ukraine \\
Tereshchenkivska str., 3, Kyiv, 01024, Ukraine}
\email{maks@imath.kiev.ua}


\newcommand\mycolor[1]{}

\setlist[enumerate]{itemsep=0.3ex, topsep=0.3ex, label={\rm(\arabic*)}}
\setlist[itemize]{itemsep=0.3ex, topsep=0.3ex, leftmargin=4ex}

\newtheorem{theorem}[subsection]{Theorem}
\newtheorem{lemma}[subsection]{Lemma}

\newtheorem{corollary}[subsection]{Corollary}

\newtheorem{remark}[subsection]{Remark}

\makeatletter
\@addtoreset{subsection}{section}
\@addtoreset{equation}{section}
\@addtoreset{figure}{section}
\@addtoreset{table}{section}
\makeatother

\makeatletter
\newcommand\testshape{family=\f@family; series=\f@series; shape=\f@shape.}
\def\myemphInternal#1{\if n\f@shape%
\begingroup\itshape #1\endgroup\/%
\else\begingroup\sf\itshape\small #1\endgroup%
\fi}
\def\myemph{\futurelet\testchar\MaybeOptArgmyemph}
\def\MaybeOptArgmyemph{\ifx[\testchar \let\next\OptArgmyemph
                 \else \let\next\NoOptArgmyemph \fi \next}
\def\OptArgmyemph[#1]#2{\index{#1}\myemphInternal{#2}}
\def\NoOptArgmyemph#1{\myemphInternal{#1}}
\makeatother

\newcommand\term[2][\empty]{\myemph[#1]{#2}}

\newcommand\amatr[4]{\left(\!\begin{smallmatrix}#1\ &#2\\[0.5mm] #3\ &#4 \end{smallmatrix}\!\right)}

\newcommand\Aman{A}

\newcommand\Cman{C}

\newcommand\Kman{K}

\newcommand\Mman{M}

\newcommand\Tman{T}
\newcommand\Uman{U}

\newcommand\Xman{X}
\newcommand\Yman{Y}

\newcommand\bC{\mathbb{C}}

\newcommand\bR{\mathbb{R}}
\newcommand\bZ{\mathbb{Z}}

\newcommand\id{\mathrm{id}}          
\newcommand\eps{\varepsilon}                   
\newcommand\ntimes{\widetilde{\times}}         
\newcommand\GL{\mathrm{GL}}
\newcommand\SL{\mathrm{SL}}

\newcommand\Diff{\mathcal{D}}       
 
\newcommand\DiffId{\Diff_{\id}}     

\newcommand\Cinfty{\mathcal{C}^{\infty}}

\newcommand\Ci[2]{\mathcal{C}^{\infty}(#1,#2)}               

\newcommand\fixsymbol{\mathrm{fix}}
\newcommand\invsymbol{}
\newcommand\nbsymbol{\mathrm{nb}}
\newcommand\folsymbol{*}

\newcommand\DiffInv[3][\empty]{\Diff_{\invsymbol}(#2,#3\ifx\empty #1\relax\else,#1\fi)}

\newcommand\DiffFix[3][\empty]{\Diff_{\fixsymbol}(#2,#3\ifx\empty #1\relax\else,#1\fi)}

\newcommand\DiffNb[3][\empty]{\Diff_{\nbsymbol}(#2,#3\ifx\empty #1\relax\else,#1\fi)}

\newcommand\DiffHFix[3][\empty]{\Diff^{0}_{\fixsymbol}(#2,#3\ifx\empty #1\relax\else,#1\fi)}

\newcommand\DiffHNb[3][\empty]{\Diff^{0}_{\nbsymbol}(#2,#3\ifx\empty #1\relax\else,#1\fi)}

\newcommand\DiffPlusFix[3][\empty]{\Diff^{+}_{\fixsymbol}(#2,#3\ifx\empty #1\relax\else,#1\fi)}

\newcommand\FDiff[2][\empty]{\Diff^{\folsymbol}(#2\ifx\empty #1\relax\else,#1\fi)}
\newcommand\FDiffFix[2][\empty]{\Diff^{\folsymbol}_{\fixsymbol}(#2\ifx\empty #1\relax\else,#1\fi)}
\newcommand\FDiffA[2][\empty]{\Diff^{=}(#2\ifx\empty #1\relax\else,#1\fi)}

\newcommand\VBAut[2][\empty]{\GL(#2\ifx\empty #1\relax\else,#1\fi)}

\newcommand\DiffLP{\Diff}  
\newcommand\DiffLPInv[3][\empty]{\DiffLP_{inv}(#2,#3\ifx\empty#1\relax\else,#1\fi)}

\newcommand\DiffLPFix[3][\empty]{\DiffLP_{fix}(#2,#3\ifx\empty#1\relax\else,#1\fi)}

\newcommand\DiffLPNb[3][\empty]{\DiffLP_{nb}(#2,#3\ifx\empty#1\relax\else,#1\fi)}

\newcommand\func{f}
\newcommand\gfunc{g}
\newcommand\dif{h}
\newcommand\gdif{g}
\newcommand\qdif{q}

\newcommand\px{x}
\newcommand\py{y}
\newcommand\pz{z}

\newcommand\Circle{S^1}

\newcommand\Klein{\Kman}
\newcommand\SKlein{\mathbf{\Kman}}
\newcommand\Klev[1]{\Kman_{#1}}

\newcommand\STorus{\mathbf{T}}
\newcommand\Tlev[1]{\Tman_{#1}}

\newcommand\Lpq[2]{L_{#1,#2}} 

\newcommand\FolDiff{\Diff^{fol}}    
\newcommand\FolLpDiff{\Diff^{lp}}   

\newcommand\Foliation{\mathcal{F}}
\newcommand\GFoliation{\mathcal{G}}

\newcommand\vbp{{\mycolor{blue}p}}  

\newcommand\az{{\mycolor{red}\pz}}

\newcommand\xtor[2]{{}^{#1}_{#2}}

\newcommand\xar[2]{ \ar@{=>}[r]^-{#1}_-{#2}}

\newcommand\diflpdir[4][1.4em]{\xymatrix@C=#1{#2: \xtor{0}{1} \xar{#3}{#4} & \xtor{0}{1}}}
\newcommand\diflprew[4][1.4em]{\xymatrix@C=#1{#2: \xtor{0}{1} \xar{#3}{#4} & \xtor{1}{0}}}

\begin{document}
\begin{abstract}
Let $\mathcal{G}$ be a Morse-Bott foliation on the solid Klein bottle $\mathbf{K}$ into $2$-dimensional Klein bottles parallel to the boundary and one singular circle $S^1$.
Let also $S^1\widetilde{\times}S^2$ be the twisted bundle over $S^1$ which is a union of two solid Klein bottles $\mathbf{K}_0$ and $\mathbf{K}_1$ with common boundary $K$.
Then the above foliations $\mathcal{G}$ on both $\mathbf{K}_0$ and $\mathbf{K}_1$ gives a foliation $\mathcal{G}'$ on $S^1\widetilde{\times}S^2$ into parallel Klein bottles and two singluar circles.
The paper computes the homotopy types of groups of foliated (sending leaves to leaves) and leaf preserving diffeomorphisms for foliations $\mathcal{G}$ and $\mathcal{G}'$.
\end{abstract}

\maketitle

\newcommand\prFromPlaneToDisk{{\color{red}\alpha}}   
\newcommand\prFromDiskToKlein{{\color{red}\beta}}     
\newcommand\prFromKleinToCircle{{\color{red}\gamma}}   

\newcommand\prFromDiskToCircle{{\color{red}\kappa}}
\newcommand\prFromDiskToR{{\color{red}\lambda}}
\newcommand\prRToCircle{{\color{red}\sigma}}   
\newcommand\prFromPlaneToKlein{{\color{red}\zeta}}   

\newcommand\ha{\mathbf{a}}
\newcommand\hb{\mathbf{b}}

\newcommand\aaprod[2]{#2\times#1}
\newcommand\naaprod[2]{#2\,\ntimes\,#1}

\newcommand\Disk{D^2}

\newcommand\CxS{\aaprod{\bC}{\Circle}}
\newcommand\CnxS{\naaprod{\bC}{\Circle}}
\newcommand\CxR{\aaprod{\bC}{\bR}}
\newcommand\CnxR{\naaprod{\bC}{\bR}}
\newcommand\SSnxS{\naaprod{S^2}{\Circle}}
\newcommand\DxR{\aaprod{\Disk}{\bR}}
\newcommand\SxR{\aaprod{\Circle}{\bR}}
\newcommand\ZxR{\aaprod{0}{\bR}}
\newcommand\DzxR{\aaprod{(D^2\setminus\{0\})}{\bR}}
\newcommand\CzxR{\aaprod{(\bC\setminus\{0\})}{\bR }}

\newcommand\tfib[1]{\mathsf{T}_{\!\mathrm{fib}}(\dif)}

\newcommand\kTwist{\tau}
\newcommand\kMu{\mu}
\newcommand\kLambda{\lambda}
\newcommand\tTwist{\widehat{\kTwist}}
\newcommand\tMu{\widehat{\kMu}}
\newcommand\tLambda{\widehat{\kLambda}}

\newcommand\at{{\color{red}\tau}}
\newcommand\aw{{\color{blue}w}}
\newcommand\arad{{\color{blue}r}}
\newcommand\as{{\color{blue}s}}
\newcommand\aphi{{\color{blue}\phi}}

\newcommand\tfunc{\widehat{\func}}
\newcommand\tgfunc{\gfunc'}   
\newcommand\ttgfunc{\gfunc''} 
\newcommand\tdif{\dif'}       
\newcommand\ttdif{\dif''}     

\newcommand\liftFunc[1]{\lambda_{#1}}

\newcommand\Autfx{\mathcal{A}}
\newcommand\tAutfx{\widehat{\Autfx}}
\newcommand\PAutfx{\mathcal{A}^{+}}
\newcommand\tPAutfx{\widehat{\Autfx}^{+}}

\newcommand\qA{{\color{orange}\sigma}}
\newcommand\fGrp[1]{\mathcal{B}_{#1}}

\newcommand\fGTotal{\fGrp{ 0 }}
\newcommand\fGFixZR{\fGrp{ 1 }}
\newcommand\fGLin{\fGrp{ 2 }}

\newcommand\fGFixSandw{\fGrp{ 3 }}
\newcommand\fGFixCollar{\fGrp{ 4 }}

\newcommand\TFoliation{\mathcal{F}}
\newcommand\KFoliation{\mathcal{G}}
\newcommand\SFoliation{\widehat{\KFoliation}}

\newcommand\DLpKdK{\FolLpDiff(\KFoliation,\partial\SKlein)}
\newcommand\DLpK{\FolLpDiff(\KFoliation)}
\newcommand\DFolKdk{\FolDiff(\KFoliation,\partial\SKlein)}
\newcommand\DFolK{\FolDiff(\KFoliation)}

\newcommand\DiffRCommShift{\Diff_{\eta}^{+}(\bR)}

\newcommand\RCircle[1]{S_{#1}}   
\newcommand\RDisk[1]{D_{#1}}     
\newcommand\RTD[1]{T_{#1}}                
\newcommand\RSkl[1]{\mathbf{K}_{#1}}      
\newcommand\RBd[1]{\mathbf{C}^{#1}}       
\newcommand\XC{\Klev{0}}

\newcommand\aConst{0.2}
\newcommand\bConst{0.8}

\newcommand\Hhom{H}
\newcommand\Ghom{G}
\newcommand\nrm[1]{\vert#1\vert}

\newcommand\vba{a}
\newcommand\vbb{b}
\newcommand\vbc{c}

\newcommand\CRRaa{\Cinfty_{*}(\bR,\bR)}

\section{Introduction}
Let $D^2 = \{|\aw|\leq 1\}$ be the unit disk in the complex plane, $\Circle=\partial D^2$ be its boundary, and $\STorus = \Circle\times D^2$ be the solid torus.
Define the following $\Cinfty$ function $\func:\STorus\to[0;1]$ by $\func(\az,\aw) = \nrm{\aw}^2$, and for every $\arad\in[0;1]$ let $\Tlev{\arad}:=\func^{-1}(\arad)$ be the inverse image of $\arad$.
Evidently, $\Tlev{\arad}$ is a $2$-torus for $\arad\in(0;1]$ and $\Tlev{0}$ is a circle.
Denote by $\TFoliation = \{ \Tlev{\arad}\mid\arad\in[0;1]\}$ the partition on $\STorus$ into the inverse images of $\func$.

Notice that $\func$ is a \term{Morse-Bott} function (in some sense the most simplest one), and the corresponding partition $\TFoliation$ is a \term{Morse-Bott} foliation, see e.g.~\cite{Bott:IHES:1988, MartinezAlfaroMezaSarmientoOliveira:TMNA:2018, EvangelistaSuarezTorresVera:JS:2019}.
Such foliations play and important role in Hamiltonian and Poisson geometries, however in the present paper we will not use this interpretation.

Given a partition $\Foliation$ on a manifold $\Mman$ we will say that a diffeomorphism $\dif:\Mman\to\Mman$ is \term{$\Foliation$-leaf preserving} if it leaves invaraint each leaf of $\Foliation$, i.e.\ $\dif(\omega)=\omega$ for all $\omega\in\Foliation$.
Also, $\dif$ is \term{$\Foliation$-foliated} whenever the image $\dif(\omega)$ of each leaf $\omega\in\Foliation$ is again a (perhaps some other) leaf of $\Foliation$.
Then for a subset $\Xman\subset\Mman$ we will denote by $\FolLpDiff(\Foliation,\Xman)$ and $\FolDiff(\Foliation,\Xman)$ respectively the groups of $\Foliation$-leaf preserving and $\Foliation$-foliated diffeomorphisms of $\Mman$.
If $\Xman=\varnothing$, we will omit it from notations.
Evidently, $\FolLpDiff(\Foliation,\Xman)$ is a normal subgroup of $\FolDiff(\Foliation,\Xman)$.

In a series of previous papers by the author and O.~Khokhliuk~\cite{KhokhliukMaksymenko:IndM:2020, KhokhliukMaksymenko:PIGC:2020, KhokhliukMaksymenko:fol_nbh:2022, KhokhliukMaksymenko:lens:2022, Maksymenko:lens:2023} there were computed homotopy types of groups of foliated and leaf preserving diffeomorphisms of the above foliation $\TFoliation$ on $\STorus$.
Namely, the following results are obtained:
\begin{theorem}[{\cite[Theorem~1.1.1]{KhokhliukMaksymenko:lens:2022}}]\label{th:STorus:DiffLp_contr}
The group $\FolLpDiff(\TFoliation,\partial\STorus)$ is weakly contractible.
\end{theorem}
\begin{theorem}[{\rm\cite{Maksymenko:lens:2023}}]\label{th:STorus:Lp_is_def_retr_Fol}
The pair $\bigl(\FolLpDiff(\TFoliation), \FolLpDiff(\TFoliation,\partial\STorus) \bigr)$ is a strong deformation retract of the pair $\bigl(\FolDiff(\TFoliation), \FolDiff(\TFoliation,\partial\STorus) \bigr)$.
\end{theorem}
It is also known that the group $\Diff(\STorus,\partial\STorus)$ of diffeomorphisms of $\STorus$ fixed on its boundary is contractible (N.~Ivanov~\cite{Ivanov:ZNSL:1976}), while the group of all diffeomorphisms $\Diff(\STorus)$, contains as a deformation retract a certain semidirect product $\mathcal{A}:=(\Circle\times\Circle)\rtimes\Uman$, where $\Uman = \{ \amatr{\eps}{n}{0}{\delta} \mid \eps,\delta\in\{\pm1\}, n\in\bZ \} \subset \SL(2,\bZ)$ (this is a classical result).
In particular, $\pi_0\Diff(\STorus) \cong \Uman$.
In fact, $\mathcal{A}$ is contained even in $\FolLpDiff(\TFoliation,\partial\STorus)$, and Theorems~\ref{th:STorus:DiffLp_contr} and~\ref{th:STorus:Lp_is_def_retr_Fol} imply that the following inclusions are weak homotopy equivalences:
\begin{gather*}
\xymatrix@R=0.7ex{
    \ \{\id_{\STorus}\}                       \  \ar@{^(->}[r] &
    \ \FolLpDiff(\TFoliation,\partial\STorus) \  \ar@{^(->}[r] &
    \ \FolDiff(\TFoliation,\partial\STorus)   \  \ar@{^(->}[r] &
    \ \Diff(\STorus,\partial\STorus),         \    \\
    \ \mathcal{A}                             \  \ar@{^(->}[r] &
    \ \FolLpDiff(\TFoliation)                 \  \ar@{^(->}[r] &
    \ \FolDiff(\TFoliation)                   \  \ar@{^(->}[r] &
    \ \Diff(\STorus).                         \
}
\end{gather*}
Furthermore, gluing two copies $\STorus_0$ and $\STorus_1$ of the solid torus by some diffeomorphisms between their boundaries, one gets a lens space $\Lpq{p}{q}$.
Then the foliation $\TFoliation$ on each of those tori gives a foliation $\TFoliation_{p,q}$ on $\Lpq{p}{q}$ into two singluar circles and parallel $2$-tori.
In~\cite{KhokhliukMaksymenko:lens:2022, Maksymenko:lens:2023} there were also computed the homotopy types of the groups $\FolLpDiff(\TFoliation_{p,q})$ and $\FolDiff(\TFoliation_{p,q})$.

In the present paper we will make similar computations for the non-orientable counterpatrs of $\STorus$ and lens spaces: the solid Klein bottle $\SKlein$ and the twisted $S^2$-bundle over the circle $\SSnxS$.

More precisely, consider the following orientation reversing diffeomorphism $\xi:\STorus\to\STorus$ of order $2$ given by $\xi(\aw,\az)=(\bar{\aw}, -\az)$.
Then the quotient space $\SKlein := \STorus/\xi$ is called the \term{solid Klein bottle}.
It is a non-orientable $3$-manifold and the corresponding quotient map $\vbp:\STorus\to\SKlein$ is its orientable double covering.
Moreover, it is evident that $\func\circ\xi=\func$, whence there exists a $\Cinfty$ function $\gfunc:\SKlein\to[0;1]$ such that $\func = \gfunc\circ\vbp$.
For $t\in[0;1]$ let $\Klev{t} = \gfunc^{-1}(t)$ be the corresponding level set of $\gfunc$, so $\XC$ is a circle, while for $t\in(0;1]$ the set $\Klev{t}$ is a ($2$-dimensional) Klein bottle.
In particular, $\Klev{1} = \partial\SKlein$.

Note also that $\Tlev{\arad} = \vbp^{-1}(\Klev{t})$ and the restriction maps $\vbp:\Tlev{\arad} \to \Klev{t}$ (oriented for $\arad>0$) double coverings.

Let $\KFoliation = \{ \Klev{t} \}_{t\in[0;1]}$ be the partition of $\SKlein$ into level sets of $\gfunc$.
Then $\gfunc$ is a Morse-Bott function for which $\XC$ is a non-degenerate critical manifold, and all other points of $\SKlein$ are regular for $\gfunc$.
Therefore one can regard $\KFoliation$ as a \term{Morse-Bott} foliation with the singluar leaf $\XC$.

Our aim is to compute the homotopy types of the groups $\DFolK$ and $\DLpK$ of $\KFoliation$-foliated and $\KFoliation$-leaf preserving diffeomorphism of $\SKlein$ respectively, and their respective subgroups fixed on $\partial\SKlein$.
In fact most of the preliminary work is done in the mentioned above papers, and here we will just use their results for explicit computations.
We are also aimed here to illustrate usefulness of the developed methods.

First we recall the following statement:
\begin{lemma}\label{lm:DiffK}
$\Diff(\SKlein,\partial\SKlein)$ is contractible, while $\Diff(\SKlein)$ is homotopy equivalent to $\Circle\times\bZ_2\times\bZ_2$, i.e.\ to the disjoint union of $4$ circles.
\end{lemma}
\begin{proof}[History of proof]
W.~B.~R.~Likorish~\cite{Lickorish:PCPS:1963} shown that $\pi_0\Diff(\SKlein) \cong \bZ_{2}\oplus\bZ_{2}$ and that each diffeomorphism of $\partial\SKlein$ extends to some diffeomorphism of $\SKlein$.
The latter can be rephrased so that the restriction map $\rho:\Diff(\SKlein) \to \Diff(\partial\SKlein)$, $\rho(\dif)=\dif|_{\partial\SKlein}$, (being also a continuous homomorphism) is surjective.
Evidently, its kernel is the group $\Diff(\SKlein,\partial\SKlein)$ of diffeomorphisms of $\SKlein$ fixed on the boundary.
Notice that the map $\rho$ is known to be a locally trivial fibration which is a particular case of \term{``local triviality for embeddings''} statement independently proved by J.Cerf~\cite{Cerf:BSMF:1961}, R.~Palais~\cite{Palais:CMH:1960}, and E.~Lima~\cite{Lima:CMH:1964}.

Also, N.~Ivanov~\cite{Ivanov:ZNSL:1976} obtianed a general result on Waldhausen manifold which includes the statement that $\Diff(\SKlein,\partial\SKlein)$ is contractible.

This implies that $\rho$ is a homotopy equivalence.
Moreover, it is also proved by C.~Earle and J.~Eeels~\cite{EarleEells:JGD:1969} and A.~Gramain~\cite{Gramain:ASENS:1973} that the path components of $\Diff(\partial\SKlein)$ are homotopy equivalent to the circle.
Hence $\Diff(\SKlein)$ is homotopy equivalent to $\Circle\times\bZ_2\times\bZ_2$, i.e.\ to the disjoint union of $4$ circles.
\end{proof}

Our first result shows that the corresponding $\KFoliation$-foliated and $\KFoliation$-leaf preserving counterparts of the groups from Lemma~\ref{lm:DiffK} have the same homotopy types.
So the situation here is literally the same as in Theorems~\ref{th:STorus:DiffLp_contr} and~\ref{th:STorus:Lp_is_def_retr_Fol}.

\begin{theorem}\label{th:homtype:Kfol}
The following statements hold.
\begin{enumerate}[leftmargin=*]
\item\label{enum:th:homtype:Kfol:lp_dK_contr}
The group $\DLpKdK$ is weakly contractible.

\item\label{enum:th:homtype:Kfol:lp_fol_def_retr}
The pair $\bigl(\DLpK, \DLpKdK \bigr)$ is a strong deformation retract of $\bigl(\DFolK, \DLpKdK \bigr)$.
\end{enumerate}
They imply that the following maps denoted by (w.)h.e. are (weak) homotopy equivalences:
\[
    \xymatrix@R=2ex@C=6.7ex{
        \ \{\id_{\SKlein}\}                \     \ar@{^(->}[r]^-{\text{\bf w.h.e}} \ar@/^5ex/[rrr]^{\text{h.e}} &
        \ \DLpKdK                          \     \ar@{^(->}[r]^-{\text{\bf h.e}}  &
        \ \DFolKdk                         \     \ar@{^(->}[r]^-{\text{\bf w.h.e}} &
        \ \Diff(\SKlein,\partial\SKlein),                                 \
    }
\]
\[
\xymatrix@R=2ex@C=6.7ex{
    \ \DLpK                                \     \ar@{^(->}[r]^-{\text{\bf h.e}}   &
    \ \DFolK                               \     \ar@{^(->}[r]^{\text{\bf w.h.e.}}   &
    \ \Diff(\SKlein)                       \     \ar@{->>}[r]^-{\rho}_-{\text{h.e.}} &
    \ \Diff(\partial\SKlein) \simeq \mathop{\sqcup}\limits_{4} \Circle,
}
\]
where the notations in \textbf{bold} denote new results and implications.
\end{theorem}

This theorem will be proved in Section~\ref{sect:proof:th:homtype:Kfol}.

As mentioned above a \term{lens space} is a $3$-manifold obtianed by gluing two solid tori $\STorus_0$ and $\STorus_1$ by some diffeomorphism between their boundaries, and there are infinitely many mutually non-diffeomorphic lens spaces.
On the other hand, since every diffeomorphism of the Klein bottle $\Klein$ extends to a diffeomorphism of the solid Klein bottle $\SKlein$, it follows that gluing two copies of $\SKlein$ by some diffeomorphism of their boundaries gives always rise to the same manifold $\SSnxS$ being a total space of a unique non-trivial $S^2$-bundle over $\Circle$ called the \term{twisted $S^2$-bundle over $\Circle$}.

Therefore one can regard $\SSnxS$ as the union of two Solid klein bottles $\KFoliation_0$ and $\KFoliation_1$ glued by the indentity diffeomorphism of their boundaries.
On each $\SKlein_i$, $i=0,1$, we have defined above the foliation $\KFoliation$ into paralled $2$-dimensional Klein bottles and one singluar circle.
These foliations constitute together a foliation on $\SSnxS$ into parallel Klein bottles and two singluar circles $\Circle_i \subset\SKlein_i$.
We will denote that foliation by $\SFoliation$, and it will be convenient to call it \term{polar}.

As a consequence of Theorem~\ref{th:homtype:Kfol} we get the following description of the homotopy types of $\SFoliation$-foliated and $\SFoliation$-leaf preserving diffeomorphisms of $\SSnxS$.
Notice that each $\dif\in\FolLpDiff(\SFoliation)$ leaves invariant the common boundary $\partial\SKlein_0=\partial\SKlein_1$ which we will denote by $\Klein$.
Hence we have a well-defined continuous restriction homomorphism $\rho:\FolLpDiff(\SFoliation) \to \Diff(\Klein)$, $\rho(\dif)=\dif|_{\Klein}$.
Its kernel is eveidently the group $\FolLpDiff(\SFoliation,\Klein)$ of $\SFoliation$-leaf preserving diffeomorphisms fixed on $\Klein$.

Denote by $\FolDiff_{+}(\SFoliation)$ the (index $2$) subgroup of $\FolDiff(\SFoliation)$ consisting of diffeomorphisms leaving invaraint each singular circle $\Circle_0$ and $\Circle_1$.
Our second result if the following Theorem~\ref{th:homtype:S2nxS1} which will be proved in Section~\ref{sect:proof:th:homtype:S2nxS1}:
\begin{theorem}\label{th:homtype:S2nxS1}
The following statements hold.
\begin{enumerate}[label={\rm(\arabic*)}, leftmargin=*]
\item\label{enum:th:homtype:S2nxS1:lp_fol_def_retr} The group $\FolLpDiff(\SFoliation)$ is a strong deformation retract of $\FolDiff_{+}(\SFoliation)$.
\item\label{enum:th:homtype:S2nxS1:Dlp_DK} The ``restriction to $\Klein:=\partial\SKlein_0=\partial\SKlein_1$ homomorphism'' $\rho:\FolLpDiff(\SFoliation) \to \Diff(\Klein)$ is a weak homotopy equivalence.
\end{enumerate}
In particular, $\FolLpDiff(\SFoliation)$ and $\FolDiff_{+}(\SFoliation)$ are weakly homotopy equivalent to the disoint union of $4$ circles, while $\FolDiff(\SFoliation)$ is homotopy equivalent to the disjoint union of $8$ circles.
\end{theorem}

Note that M.~Kim and F.~Raymond~\cite{KimRaymond:TrAMS:1990}, shown that $\pi_0\Diff(\SSnxS) \simeq \bZ_2\oplus\bZ_2$, and the generators of that group can be chosen to be also the generators of $\pi_0\FolLpDiff(\SFoliation)$.
This gives the following
\begin{corollary}
The inclusions $\FolLpDiff(\SFoliation) \subset \FolDiff_{+}(\SFoliation) \subset \Diff(\SSnxS)$ induce bijections on $\pi_0$ groups:
\[
\pi_0\FolLpDiff(\SFoliation) \, = \, \pi_0\FolDiff_{+}(\SFoliation) \, = \, \pi_0\Diff(\SSnxS) \, = \, \bZ_2\oplus\bZ_2.
\qedhere
\]
\end{corollary}
On the other hand, the homotopy type of $\Diff(\SSnxS)$ is more complicated.
It was described in E.~C\'{e}sar and C.~Rourke~\cite[Theorem~2]{CesarRourke:BAMS:1979}, which in turn extended the technnique from PhD theses by E.~C\'{e}sar~\cite{Cesar:Theses:1977}.

\section{Twisted bundles over the circle}
In this sections we present an explicit model for the solid Klein bottle $\SKlein$ and the universal covers of $\SKlein$ and $\SKlein\setminus\XC$.
These notations will be used in the proof of Theorem~\ref{th:homtype:Kfol}.

\subsection{Universal cover of the solid Klein bottle $\SKlein$}
Let $\prFromDiskToCircle:\CxR\to\bR$, $\prFromDiskToCircle(\as,\aw)=\as$, be the trivial vector bundle (of real dimension $2$), and $\tgfunc:\CxR\to\bR$ be a $\Cinfty$ function given by $\tgfunc(\as,\aw)=|\aw|^2$.
It determines a norm (or scalar product) on the fibers of $\prFromDiskToCircle$.

Consider the following $\Cinfty$ vector bundle isomorphism $(\xi,\eta)$ reversing orientation and having no fixed points:
\begin{equation}\label{equ:xi_eta}
\begin{gathered}
\xymatrix@C=3cm{
    \CxR \ar[r]^{\xi:\, (\as,\,\aw)\,\mapsto\,(\as+1,\,\bar{\aw})} \ar[d]_{\prFromDiskToCircle} &
    \CxR \ar[d]^{\prFromDiskToCircle} \\
    \bR \ar[r]^{\eta:\, \as\,\mapsto\,\as+1} &
    \bR
}
\end{gathered}
\end{equation}
It defines a free and properly discontinuous action of $\bZ$ on $\CxR$.
Denote by $\CnxS = (\CxR)/\bZ$ the quotient space.
Then $\bR/\eta\cong\Circle$ and the corresponding quotient maps $\prFromDiskToKlein:\CxR\to\CnxS$ and $\prRToCircle:\bR\to\Circle$, $\prRToCircle(\as) = e^{2\pi i \as}$, are universal coverings maps.
Moreover, we get the well-defined quotient vector bundle $\prFromKleinToCircle:\CnxS \equiv (\CxR)/\xi \to \bR/\eta \equiv \Circle$ such that the following diagram is commutative:
\[
\xymatrix{
    \CxR \ar[r]^-{\prFromDiskToKlein} \ar[d]_-{\prFromDiskToCircle} &
    \CnxS \ar[d]^-{\prFromKleinToCircle} \\
    \bR \ar[r]^-{\prRToCircle} &
    \Circle
}
\]
Evidently, $\tgfunc\circ\xi=\tgfunc$, whence there exists a unique $\Cinfty$ function $\gfunc:\CnxS\to\bR$, such that $\tgfunc = \gfunc\circ\prFromDiskToKlein$.
Then $\SKlein:=\gfunc^{-1}\bigl([0;1]\bigr)$ is the solid Klein bottle, $\prFromDiskToKlein: \DxR \to \SKlein$ is the universal cover of $\SKlein$, and $\KFoliation = \{ \Klev{\arad}:=\gfunc^{-1}(\arad)\}_{\arad\in[0;1]}$ consists of level sets of $\gfunc$.
We also that we have a norm on the fibers $\nrm{\cdot}:\CnxS \to[0;+\infty]$ given by $\nrm{\px} = \sqrt{\gfunc(\px)}$ for $\px\in\CnxS$, so $\Klev{\arad}$ consists of elements of $\CnxS$ of norm $\sqrt{\arad}$.

For $\arad\in[0;1]$ denote:
\begin{align*}
    \RSkl{\arad} &:= \gfunc^{-1}\bigl( [0;\arad] \bigr) = \{ \px\in\CnxS \mid \nrm{\px} \leq \sqrt{\arad} \}, &
    \RBd{\arad}  &:= \gfunc^{-1}\bigl( [\arad;1] \bigr).
\end{align*}
Thus $\RSkl{\arad}$ is a ``thinner'' solid Klein bottle with boundary $\Klev{\arad}$, while $\RBd{\arad}$ is a collar of the boundary Klein bottle $\partial\SKlein$.
In particular, $\XC$ is a circle being also the zero section of $\prFromKleinToCircle$.

\subsection{Universal cover of $\SKlein\setminus\XC$}
Consider also another universal covering map
\[
    \prFromPlaneToDisk:\bR^2\times(0;+\infty) \to \CzxR,
    \qquad
    \prFromPlaneToDisk(\as,\aphi,\arad) = (\as, \arad e^{2\pi i \phi}).
\]
Then the composition:
\[
    \prFromPlaneToKlein:
    \bR^2\times(0;1]  \xrightarrow{~\prFromPlaneToDisk~}
    \DzxR \xrightarrow{~\prFromDiskToKlein~}
    \SKlein\setminus\XC
\]
is the universal covering map for $\SKlein\setminus\XC$.
Evidently, for each leaf $\Klev{\arad}$, $\arad>0$, its inverse $\prFromPlaneToKlein^{-1}(\Klev{\arad})$ is the horizontal plane $\bR^2\times\arad$.
In other words, the composition $\ttgfunc:= \tgfunc\circ\prFromPlaneToDisk:\bR^2\times(0;+\infty)\to(0;+\infty)$ is just the coordinate projection, $\ttgfunc(\as,\aphi,\arad)=\arad$.

One easily checks that the group of covering translations of $\prFromPlaneToKlein$ is generated by the following diffeomorphisms
\begin{equation}\label{equ:ha_hb}
\begin{gathered}
    \ha,\hb:\bR^2\times(0+\infty) \to \bR^2\times(0+\infty), \\
\begin{aligned}
    \ha(\as,\aphi,\arad) &= (\as, \aphi+1, \arad), &\qquad
    \hb(\as,\aphi,\arad) &= (\as+1, -\aphi, \arad),
\end{aligned}
\end{gathered}
\end{equation}
and that the following identities holds:
\begin{align*}
    &\prFromPlaneToDisk\circ\ha=\prFromPlaneToDisk,&
    &\prFromPlaneToDisk\circ\hb=\xi\circ\prFromPlaneToDisk,&
    &\hb\circ\ha = \ha^{-1}\circ\hb.
\end{align*}

Let us collect all those spaces and maps into the following commutative diagram:
\begin{equation}\label{equ:coverings}
\begin{gathered}
\xymatrix{
    \ \bR^2\times(0;1]               \
    \ar[d]_-{\prFromPlaneToDisk}
    \ar@{^(->}[rr]
    \ar@/_10ex/[dd]_-{\prFromPlaneToKlein} &&
    \ \bR^2\times(0;+\infty) \
    \ar[d]_-{\prFromPlaneToDisk}
    \ar@/^1ex/[rdd]^-{\ttgfunc}
    \\
    \ \DzxR                             \ \ar@{^(->}[r] \ar[d]_-{\prFromDiskToKlein} &
    \ \DxR                              \ \ar@{^(->}[r] \ar[d]_-{\prFromDiskToKlein} &
    \ \CxR                              \ \ar[d]_-{\prFromDiskToKlein} \ar[dr]^-{\tgfunc} \\
    \ \SKlein\setminus\XC               \  \ar@/_2ex/[dr]^-{\prFromKleinToCircle}  \ar@{^(->}[r] &
    \ \SKlein=\gfunc^{-1}([0;1])        \ \ar[d]^-{\prFromKleinToCircle} \ar@{^(->}[r] &
    \ \CnxS                             \ \ar@/^2ex/[dl]_-{\prFromKleinToCircle} \ar[r]^-{\gfunc} &
    \ \bR                               \ \\
    &
    \ \Circle \
}
\end{gathered}
\end{equation}

\begin{figure}[ht]
\includegraphics[width=7cm]{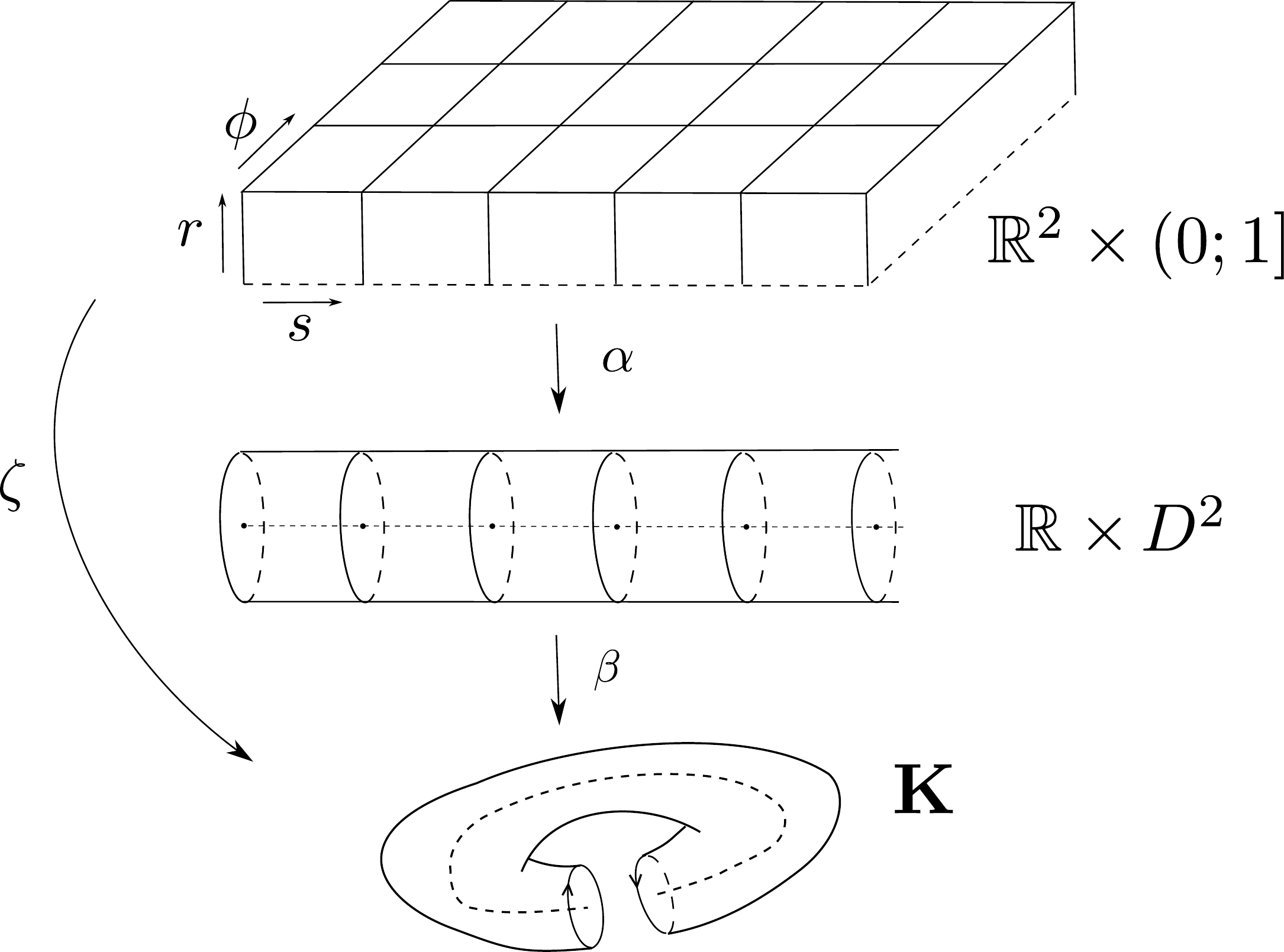}
\caption{Universal coverings of $\SKlein$ and $\SKlein\setminus\XC$}
\label{fig:all_coverings}
\end{figure}

\subsection{Liftings of diffemorphisms}
We will also use the following simple statement concerning covering maps:
\begin{lemma}\label{lm:lifts_commuting_with_covering_slices}
Let $\beta:\Xman\to\Yman$ be a covering map between path connected topological spaces, and $\xi:\Xman\to\Xman$ be any covering transformation, i.e.\ a homeomorphism satisfying $\beta\circ\xi=\beta$.
Let also $\dif:\Yman\to\Yman$ be a continuous map having a lifting $\tdif:\Xman\to\Xman$, so $\beta\circ\tdif = \dif\circ\beta$.
Then each of the following conditions implies that $\tdif$ commutes with $\xi$, i.e.\ $\xi\circ\tdif=\tdif\circ\xi$.
\begin{enumerate}[label={\rm(\arabic*)}]
\item\label{enum:lm:lifts_commuting_with_covering_slices:at_pt}
There exists a point $\px\in\Xman$ such that $\xi\circ\tdif(\px)=\tdif\circ\xi(\px)$.
\item\label{enum:lm:lifts_commuting_with_covering_slices:on_a_set}
There exists a subset $\Aman\subset\Yman$ being invariant under $\xi$, i.e.\ $\xi(\Aman)=\Aman$, and such that $\tdif$ is fixed on $\Aman$.
\end{enumerate}
\end{lemma}
\begin{proof}
\ref{enum:lm:lifts_commuting_with_covering_slices:at_pt}
Notice that $\xi\circ\tdif$ and $\tdif\circ\xi$ are also liftings of $\dif$, since
\begin{align*}
    &\xi\circ\tdif\circ\beta =  \xi\circ\beta\circ\dif = \beta\circ\dif,&
    &\tdif\circ\xi\circ\beta =  \tdif\circ\beta = \beta\circ\dif.
\end{align*}
Moreover, by assumption, they coincide at $\px$.
Then by uniqueness of liftings with one prescribed value, they should identically coincide on all of $\Xman$.

\ref{enum:lm:lifts_commuting_with_covering_slices:on_a_set}
Let $\px\in\Aman$ be any point.
Then, by assumption, $\xi(\px)\in\Aman$.
As $\tdif$ is fixed on $\Aman$ we have that $\tdif\circ\xi(\px) = \xi(\px) = \xi\circ\tdif(\px)$.
Hence by~\ref{enum:lm:lifts_commuting_with_covering_slices:at_pt}, $\tdif\circ\xi \equiv \xi\circ\tdif$ on all of $\Xman$.
\end{proof}

Notice that each $\dif\in\Diff(\SKlein,\partial\SKlein)$ lifts to a unique diffeomorphism $\tdif:\DxR\to \DxR$ fixed on $\SxR$, so $\prFromDiskToKlein\circ\tdif = \dif\circ\prFromDiskToKlein$.
Since $\SxR$ is invariant under $\xi$, it follows from Lemma~\ref{lm:lifts_commuting_with_covering_slices} that $\tdif$ commutes with $\xi$.

Moreover, suppose in addition that $\dif(\XC)=\XC$, $\tdif(\ZxR)=\ZxR$, then the restriction $\tdif|_{\DzxR}$ lifts in turn to a unique diffeomorphism $\ttdif:\bR^2\times(0;1]\to\bR^2\times(0;1]$ fixed on $\bR^2\times1$, so
\[ \prFromPlaneToDisk\circ\ttdif = \tdif\circ\prFromPlaneToDisk:\bR^2\times(0;1] \to \DzxR. \]
Again, since $\bR^2\times1$ is invariant under $\ha$ and $\hb$, we get from Lemma~\ref{lm:lifts_commuting_with_covering_slices} that $\ttdif$ commutes with $\ha$ and $\hb$.
These liftings $\tdif$ and $\ttdif$ of $\dif$ will play an important role for our proofs.

Let us also mention that the group $\DLpKdK$ can be defined as the subgroup of $\Diff(\SKlein,\partial\SKlein)$ consisting of diffeomorphisms preserving the function $\gfunc$, i.e.\ satisfying the identity: $\gfunc\circ\dif=\gfunc$.
Moreover, if $\dif\in\DLpKdK$, both liftings $\tdif$ and $\ttdif$ are defined and they satisfy the following idenitities: $\tgfunc\circ\tdif=\tgfunc$ and $\ttgfunc\circ\ttdif=\ttgfunc$.

\section{Proof of Theorem~\ref{th:homtype:Kfol}}\label{sect:proof:th:homtype:Kfol}
The second statement \ref{enum:th:homtype:Kfol:lp_fol_def_retr} that the pair $\bigl(\DLpK, \DLpKdK \bigr)$ is a strong deformation retract of $\bigl(\DFolK, \DLpKdK \bigr)$ is a particular case of results from~\cite{Maksymenko:lens:2023}.

For the proof of~\ref{enum:th:homtype:Kfol:lp_dK_contr} we need to define an explicit model for the solid Klein bottle $\SKlein$ and the universal covers of $\SKlein$ and $\SKlein\setminus\XC$.

\ref{enum:th:homtype:Kfol:lp_dK_contr}
The proof of contractibility of the group $\DLpKdK$ almost literally follows the proof of main result from~\cite{KhokhliukMaksymenko:lens:2022} for similar foliation on the solid torus, since the results of that paper are proved in a greater generality and are applicable in the current situation.
For the convenience of the reader we will repeat certain arguments.
Namely, we will define four subgroups of $\DLpKdK$:
\[
    \fGFixCollar      \ \subset \
    \fGFixSandw       \ \subset \
    \fGLin            \ \subset \
    \fGFixZR          \ \subset \
    \fGTotal = \DLpKdK
\]
and show%
\footnote{We also ``reorder'' here the groups $\fGrp{i}$ in comparison with their counterparts from~\cite{KhokhliukMaksymenko:lens:2022}.
Namely, we will make our diffeomorphisms fixed on the collar of $\partial\SKlein$ at the fourth step, while in~\cite{KhokhliukMaksymenko:lens:2022} that was done from the beginning.
This will slightly simplify the exposition.}
that \term{all the inclusions are homotopy equivalences, while the smallest group $\fGFixCollar$ is weakly contractible}.

To proceed with the proof it will be convenient to fix some $\Cinfty$ function $\mu:[0;1]\to[0;1]$ such that $\mu=0$ on $[0;\aConst]$ and $\mu=1$ on $[0;\bConst]$.

\begin{enumerate}[wide, label={\rm\arabic*)}, topsep=1ex, itemsep=1ex]

\item
{\bf Inclusion $\fGFixZR\subset\fGTotal$}.

Let us define the group $\fGFixZR$.
Let $\dif\in\fGTotal = \DLpKdK$, and $\tdif:\DxR\to\DxR$ be its unique lifting fixed on $\SxR$.
Since $\dif(\XC)=\XC$, it follows that $\tdif(\ZxR)=\ZxR$, so there exists an orientation preserving diffeomorphism $\sigma(\dif):\bR \to \bR$ commuting with $\xi$ and such that
\[
    \tdif(\as,0) = \bigl(\sigma(\dif)(\as), \ 0 \bigr), \ \as\in\bR.
\]
is $\ZxR$ is also invariant under $\xi$.
In particular,
\begin{align*}
    &\tdif\circ\xi(\as, 0) =\tdif(\as+1,           0) = \bigl(\sigma(\dif)(\as+1),  0 \bigr), \\
    &\xi\circ\tdif(\as, 0) =\xi(\sigma(\dif)(\as), 0) = \bigl(\sigma(\dif)(\as)+1,  0 \bigr),
\end{align*}
so $\sigma(\dif)(\as+1)=\sigma(\dif)(\as)+1$ for all $\as\in\bR$.
Let $\DiffRCommShift$ be the group of all orientation preserving diffeomorphisms $\qdif$ of $\bR$ satisfying the identity $\qdif(\as+1)=\qdif(\as)+1$ for all $\as\in\bR$.
Then the correspondence $\dif\mapsto \sigma(\dif)$ is a well-defined map $\sigma: \fGTotal \to \DiffRCommShift$.
One easily check that $\sigma$ is a continuous homomorphism.

Let $\fGFixZR := \ker(\sigma)$ be the kernel of $\sigma$.
Thus $\fGFixZR$ consists of $\KFoliation$-leaf preserving diffeomorphisms $\dif$ such that their unique lifting $\tdif$ to the universal cover $\DxR$ fixed on $\SxR$ is also fixed on $\ZxR$.

We claim that \term{$\fGFixZR$ is a strong deformation retract of $\fGTotal$}.
The proof it easy and consists of two statements.
Let us recall the arguments from~\cite{KhokhliukMaksymenko:lens:2022}.
\begin{enumerate}[label={\rm\alph*)}, wide]
\item\label{enum:def_G1_G2:section}
\textit{$\sigma$ admits a continuous section $s:\DiffRCommShift \to \DLpKdK = \fGTotal$ satisfying $s(\id_{\bR})=\id_{\SKlein}$.}
Thus $s$ is a continuous map (not necessarily a homomorphism) such that $\sigma\circ s(\gdif) = \gdif$ for all $\gdif\in\DiffRCommShift$.

Indeed, note that each $\qdif\in\DiffRCommShift$ extends to a diffeomorphism
\[
    \hat{\qdif}:\DxR\to \DxR,
    \qquad
    \hat{\qdif}(\aw,\as) = \bigl(\aw,   \mu(\nrm{\aw}) \qdif(\as) + (1-\mu(\nrm{\aw}) \as)
   \bigr)
\]
fixed even on the set $\vbp^{-1}(\RBd{\bConst}) = \{ (\aw,\as) \mid \nrm{\aw} \in [\bConst;1]\}$.
One easily checks that $\hat{\qdif}$ also commutes with $\xi$, and preserves the sets $\vbp^{-1}(\Klev{\arad})$, $\arad\in[0;1]$, being the inverses of the leaves of $\KFoliation$.
Hence $\hat{\qdif}$ yields a unique diffeomorphism $s(\qdif)$ of $\SKlein$ preserving the leaves of $\KFoliation$ and fixed on $\RBd{\bConst}$.
In other words, $s(\qdif) \in \DLpKdK = \fGTotal$.
Moreover, $\hat{\qdif}$ is in turn a unique lifting of $s(\qdif)$ fixed on $\SxR$, whence the correspondence $\qdif\mapsto s(\qdif)$ is the desired section of $\sigma$.

\item\label{enum:def_G1_G2:Diff_eta_contr}
\textit{The group $\DiffRCommShift$ is convex, and therefore contractible into the point $\id_{\bR}$ via the ``stantard'' homotopy:}
\[
    H:\DiffRCommShift\times[0;1]\to\DiffRCommShift,
    \qquad
    H(\qdif,t) = (1-t)\qdif + t\,\id_{\bR}.
\]
\end{enumerate}
Then~\ref{enum:def_G1_G2:section} allows to construct the following homeomorphism
\begin{gather*}
    \eta:\fGFixZR     \times \DiffRCommShift \ \equiv \
         \ker(\sigma) \times \DiffRCommShift \ \to    \
         \FolLpDiff(\KFoliation, \RBd{\bConst}),
    \qquad
    \eta(\dif, \qdif) = \dif \circ s(\qdif).
\end{gather*}
which is ``fixed on $\fGFixZR$'' in the sense that $\eta(\dif,\id_{\bR}) = \dif$ for all $\dif\in\fGFixZR$.
As $\DiffRCommShift$ is contractible into the point $\id_{\bR}$, it now follows that $\fGFixZR$ is a strong deformation retract of $\fGTotal=\FolLpDiff(\KFoliation, \RBd{\bConst})$.
We refer the reader to~\cite{KhokhliukMaksymenko:lens:2022} for the details.

\newcommand\Diskyr[2]{D_{#2}(#1)}
\item
{\bf Inclusion $\fGLin\subset\fGFixZR$}.

Define $\fGLin$ to be the subgroup of $\fGFixZR$ consisting of diffeomorphisms $\dif$ \term{coinciding with some vector bundle morphism $\qdif:\CnxS\to\CnxS$ on $\RSkl{\aConst}$}.

To see what this means consider the standard disk bundle $\prFromKleinToCircle:\SKlein\to\Circle$ from~\eqref{equ:coverings}, and for every $\py\in\Circle$ and $\arad\in(0;1]$ denote by
\[
    \Diskyr{\py}{\arad} = \prFromKleinToCircle^{-1}(\py) \cap \RSkl{\arad}
\]
the closed $2$-disk of radius $\arad$ in the fibre over $\py$.
Notice that the intersections of $\Diskyr{\py}{\arad}$ with the leaves of $\GFoliation$ constitute the partition of $\Diskyr{\py}{\arad}$ into concentric circles.
The following lemma is easy and directly follows from definitions of the above covering maps.
\begin{lemma}\label{lm:cond_B3}
Let $\dif\in\fGFixZR$ then the following conditions are equivalent:
\begin{enumerate}[topsep=1ex, itemsep=1ex, label={\rm(\arabic*)}, leftmargin=8ex]
\item\label{enum:cond_B3:h_in_b3}
$\dif\in\fGLin$;
\item\label{enum:cond_B3:h}
$\dif(\Diskyr{\py}{\aConst}) = \Diskyr{\py}{\aConst}$ for each $\py\in\Circle$ and the restriction $\dif:\Diskyr{\py}{\arad} \to \Diskyr{\py}{\arad}$ is a rotation (i.e.\ a linear isomorphism preserving concentric circles);

\item\label{enum:cond_B3:th}
there exists a $\Cinfty$ function $\liftFunc{\dif}:\bR\to\bR$ such that
\begin{align*}
&\tdif(\aw,\as) = (\aw e^{2\pi i \liftFunc{\dif}(\as)},\as),  &
&\liftFunc{\dif}(\as+1)=-\liftFunc{\dif}(\as),
\end{align*}
for all $\as\in\bR$ and $\aw\in D^2$ with $\nrm{\aw}\leq \aConst$;

\item\label{enum:cond_B3:tth}
there exists a $\Cinfty$ function $\liftFunc{\dif}:\bR\to\bR$ such that
\begin{align*}
&\ttdif(\as,\aphi,\arad) = (\as, \aphi + \liftFunc{\dif}(\as),\arad),  &
&\liftFunc{\dif}(\as+1)=-\liftFunc{\dif}(\as),
\end{align*}
for all $\as\in\bR$ and $\arad\in(0;\aConst]$;
\end{enumerate}
In this case such a function $\liftFunc{\dif}$ in~\ref{enum:cond_B3:th} and~\ref{enum:cond_B3:tth} is the same, and it is also unique.
\qed
\end{lemma}
\begin{proof}
Equivalence of conditions~\ref{enum:cond_B3:h_in_b3}-\ref{enum:cond_B3:tth} is easy.
It also follows from~\ref{enum:cond_B3:tth} that $\liftFunc{\dif}$ is uniquely determined by $\ttdif$.
\end{proof}

Now, by ``linearization theorem''~\cite{KhokhliukMaksymenko:fol_nbh:2022}, the inclusion $\fGLin \subset \fGFixZR$ is a homotopy equivalence.
Notice that in our situation $\KFoliation$ consists of level sets of a positive definite $2$-homogeneous on fibers function $\gfunc$.
In this case the proof of that ``linearization theorem'' can be simplified as it was shown in~\cite[Theorem~3.1.2]{KhokhliukMaksymenko:lens:2022}.

More precisely, the deformation of $\fGFixZR$ into $\fGLin$ can be defined as follows.
Let $\Uman$ be a neighborhood of the central circle $\XC$ (i.e.\ the zero section of $\vbp:\CnxS\to\Circle$) and $\dif:\Uman\to\CnxS$ a smooth embedding such that $\dif(\XC) = \XC$, but it is not necessarily fixed on $\XC$.
Then one can define the following vector bundle isomorphism
\[
\tfib{\dif}:\CnxS\to\CnxS,
\qquad
\tfib{\dif}(\px) = \lim\limits_{t\to0}\tfrac{1}{t} \dif(t\px),
\]
which can be regarded as a ``partial derivative'' of $\dif$ at points of $\XC$ in the direction of fibers of the vector bundle $\vbp:\CnxS\to\Circle$, see~\cite{KhokhliukMaksymenko:fol_nbh:2022}.
In particular, $\tfib{\dif}$ is well defined for all $\dif\in\fGFixZR$.
It is easy to see that if $\dif\in\fGLin$, so it coincides with some vector bundle morphism $\qdif$ on $\RSkl{\aConst}$, then $\tfib{\dif} \equiv \qdif$ on all of $\CnxS$.

Define also the following function
\[
    \phi:[0;1]\times(\CnxS)\to\bR,
    \qquad
    \phi(t,\px) = t + (1-t)\mu(\nrm{\px}).
\]
Evidently, $\phi(t,\px)=0$ exactly on the set $0\times\RSkl{\aConst}$.
Now a deformation $H:\fGFixZR\times[0;1]\to\fGFixZR$ of $\fGFixZR$ into $\fGLin$ can be given by the following formula:
\[
H(\dif,t)(\px) =
\begin{cases}
\frac{\dif(\phi(t,\px)\px)}{\phi(t,\px)}, &  (t,\px)\in\bigl([0;1]\times\SKlein\bigr) \setminus \bigl(0\times\RSkl{\aConst}\bigr), \\
\tfib{\dif}(\px), & (t,\px)\in0\times\RSkl{\aConst}.
\end{cases}
\]

Inbdeed, one can show (using Hadamard lemma) that $H(\dif,t)$ is a diffeomorphism of $\SKlein$ belonging to $\fGFixZR$ for all $(\dif,t)\in\fGFixZR\times[0;1]$, and the map $H$ is continuous with respect to the corrresponding $\Cinfty$ Whitney topologies, see~\cite[Theorem~3.1.2]{KhokhliukMaksymenko:lens:2022}.
Moreover, it is evident, that $H_1(\dif) = \dif$, and $H_0(\dif)|_{\RSkl{\aConst}} = \tfib{\dif}|_{\RSkl{\aConst}}$.
The latter means that $H_0(\dif)$ coincides with the vector bundle morphism $\tfib{\dif}$ on $\RSkl{\aConst}$, and thus $H_0(\dif)$ belongs to $\fGLin$.
Finally, if $\dif$ is already in $\fGLin$, then it easily follows from the formulas for $H$, that $H_t(\dif) = \dif = \tfib{\dif}$ on $\RSkl{\aConst}$.
Thus $H$ is a homotopy of $\fGFixZR$ which deforms $\fGFixZR$ into $\fGLin$ and leaves $\fGLin$ invariant.
This means that $H$ is a deformation of $\fGFixZR$ into $\fGLin$, and therefore the inclusion $\fGLin \subset \fGFixZR$ is a homotopy equivalence whose homotopy inverse is the map $H_0:\fGFixZR\to\fGLin$.

\item
{\bf Inclusion $\fGFixSandw\subset\fGLin$}.

Denote by $\CRRaa$ the subset of the space $\Ci{\bR}{\bR}$ consisting of functions $\delta:\bR\to\bR$ satisfying the identity $\delta(\as+1)+\delta(\as)\equiv 0$ for all $\as\in\bR$.

Recall that to each $\dif\in\fGLin$ one can associate a unique $\Cinfty$ function $\liftFunc{\dif}:\bR\to\bR$ satisfying conditions of Lemma~\ref{lm:cond_B3}.
In particular, $\liftFunc{\dif}\in\CRRaa$, and thus we get a well-defined map $\liftFunc{}:\fGLin\to\CRRaa$.
One easily checks $\liftFunc{\dif_1 \circ\dif_2} = \liftFunc{\dif_1} + \liftFunc{\dif_2}$ for all $\dif_1,\dif_2\in\fGLin$, so the correspondence $\dif\mapsto\liftFunc{\dif}$ is a continuous homomorphism of topological groups.

Let $\fGFixSandw := \ker(\liftFunc{})$ be the kernel of $\liftFunc{}$, so it consists of elements $\dif\in\fGLin$ for which $\liftFunc{\dif}\equiv 0$, i.e.\ $\ttdif$ is fixed on $\bR^2\times(0;\aConst]$.
One easily checks that the following two statements hold.
\begin{enumerate}[label={\rm\alph*)}, leftmargin=*]
\item {\em $\CRRaa$ is convex and therefore contractible.}
\item {\em The homomorphism $\liftFunc{}:\fGLin\to\CRRaa$ admits a continuous section $s:\CRRaa \to \fGLin$.}
Actually, for each $\delta\in\CRRaa$ we have the following diffeomorphism
\[
    \tdif_{\delta}:\DxR\to\DxR,
    \qquad
    \tdif_{\delta}(\as,\aw) = \bigl(\as, \aw e^{2\pi i (1-\mu(\arad)) \delta(\as)} \bigr).
\]
Evidently, it is fixed near $\SxR$, fixed on $\ZxR$, commutes with the covering translation $\xi:\DxR\to\DxR$ and preserves the function $\tgfunc$.
Hence $\tdif_{\delta}$ yields a unique diffeomorphism $\dif_{\delta}:\SKlein\to\SKlein$, which is fixed near $\partial\SKlein$ and preserves $\gfunc$.
In turn, $\tdif_{\delta}$ is lifting of $\dif_{\delta}$ and is of the form~\ref{enum:cond_B3:th} of Lemma~\ref{lm:cond_B3}.
Hence $\dif_{\delta}\in\fGLin$.
\end{enumerate}
Now, similarly to the proof for the inclusion $\fGFixZR\subset\fGTotal$, these conditions imply that $\fGFixSandw$ is a strong deformation retract of $\fGLin$.

\item\label{enum:incl_FixCollar__FixSandw}
{\bf Inclusion $\fGFixCollar\subset\fGFixSandw$}.

Let $\fGFixCollar$ be the subgroup of $\fGFixSandw$ consisting of $\KFoliation$-leaf preserving diffeomorphisms fixed on the collar $\RBd{\bConst}$.
Then \term{the inclusion $\fGFixCollar \subset \fGFixSandw$ is a homotopy equivalence}.
Indeed, the deformation $\Hhom:\fGFixSandw\times[0;1]\to\fGFixSandw$ of $\fGFixSandw$ into $\fGFixCollar$ can be given by:
\begin{equation}\label{equ:homotopy:incl_FixCollar__FixSandw}
\Hhom(\dif,t)(\px) =
\begin{cases}
\dif\bigl((t\nrm{\px} + (1-t)\mu(\nrm{\px})) \frac{\px}{\nrm{\px}} \bigr), & \nrm{\px}>0, \\
\dif(\px), & \nrm{\px}\leq \aConst.
\end{cases}
\end{equation}
Thus $\Hhom_{1}(\dif)=\dif$, $\Hhom_{0}$ is fixed on $\RBd{\bConst}$.

\item
{\bf Weak contractibility of $\fGFixCollar$}.

Similarly to the last step of the proof of~\cite[Theorem~1.1.1]{KhokhliukMaksymenko:lens:2022} the group $\fGFixCollar$ can be embedded into the loop space of the group of diffeomorphisms of $\Klein$, and that inclusion is a weak homotopy equivalence.

Namely, for each $\dif\in\fGFixCollar$ define the following path $\gamma_{\dif}:[0;1]\to\Diff(\Klein)$ given by
\[
\gamma_{\dif}(t)(\px) =
\begin{cases}
\px, & t=0, \\
\tfrac{1}{t}\dif(t\px), & t\in(0;1].
\end{cases}
\]
Since $\dif$ is fixed near $\XC$ and preserves ``parallel'' Klein bottles $\Klev{t}$, $t\in(0;1]$, it follows that $\gamma_{\dif}$ is a well-defined continuous \term{loop} such that $\gamma(t)=\id_{\Klein}$ for $t\in[0;\aConst]\cup[\bConst;1]$.

Moreover, the additional assumptions that for each $\dif\in\fGFixCollar$
\begin{itemize}[label={$- $}]
\item its lifting $\tdif$ is fixed near $\ZxR$
\item while the lifting $\ttdif$ of $\dif|_{\SKlein\setminus\XC}$ is fixed on $\bR^2\times(0;\aConst]$,
\end{itemize}
imply that $\dif$ actually represents a null-homotopic loops in $\Omega(\DiffId(\Klein))$.
Thus the correspondence $\dif\mapsto\gamma_{\dif}$ is an embedded of $\fGFixCollar$ into the path component $\Omega_0(\DiffId(\Klein), \id_{\Klein})$ of $\Omega(\DiffId(\Klein), \id_{\Klein})$ consisting of null-homotopic loops.

It is shown in~\cite[Corollary~1.10]{KhokhliukMaksymenko:PIGC:2020} that the latter inclusion $\fGFixCollar \subset \Omega_0(\DiffId(\Klein), \id_{\Klein})$ is a weak homotopy equivalence.
In particular, for $i\geq1$ we have the following isomorphisms:
\[
    \pi_i\fGFixCollar = \pi_i\Omega_0(\DiffId(\Klein), \id_{\Klein}) = \pi_{i+1}\DiffId(\Klein) = \pi_{i+1}\Circle = 0.
\]
Hence all homotopy groups of $\fGFixCollar$ vanish.
\end{enumerate}
This completes the proof of Theorem~\ref{th:homtype:Kfol}.
\qed

\begin{remark}\label{rem:DlpC08:contractible}
Formula~\eqref{equ:homotopy:incl_FixCollar__FixSandw} for the homotopy in the case~\ref{enum:incl_FixCollar__FixSandw} is also applicable for all $\dif\in\DLpK$ not only belonging to $\fGFixSandw$, and it gives a deformation of $\DLpK$ into $\FolLpDiff(\KFoliation,\RBd{\bConst})$.
In particular, \term{$\FolLpDiff(\KFoliation,\RBd{\bConst})$ is also weakly contractible}.
\end{remark}

\section{Proof of Theorem~\ref{th:homtype:S2nxS1}}\label{sect:proof:th:homtype:S2nxS1}
Let $\SSnxS$ be the twisted $S^2$-bundle over $\Circle$ glued from two copies of $\SKlein_0$ and $\SKlein_1$ by some diffeomorphism of their boundaries, and $\Klein:=\partial\SKlein_0=\partial\SKlein_1$ be their common boundary.
Let also $\SFoliation$ be the foliation into two circles and Klein bottles parallel ot $\Klein$.

Statement~\ref{enum:th:homtype:S2nxS1:lp_fol_def_retr} that the group $\FolLpDiff(\SFoliation)$ is a strong deformation retract of $\FolDiff_{+}(\SFoliation)$ is a direct consequence of results from~\cite{Maksymenko:lens:2023}.

\ref{enum:th:homtype:S2nxS1:Dlp_DK}
We should prove that the ``restriction to $\Klein$ homomorphism'' $\rho:\FolLpDiff(\SFoliation) \to \Diff(\Klein)$ is a weak homotopy equivalence.

As noted above, due to~\cite{Cerf:BSMF:1961, Palais:CMH:1960, Lima:CMH:1964}, this homomorphism is a locally trivial fibration whose fiber, $\FolLpDiff(\SFoliation,\Klein)$, is the group of diffeomorphisms fixed on $\Klein$.
Since $\rho$ is surjective, it suffices to show weak contractibility of $\FolLpDiff(\SFoliation,\Klein)$.

Let $\Cman$ be a neighborhood of $\Klein$ being a union of collars $\RBd{\bConst}$ of $\Klein$ in $\SKlein_0$ and $\SKlein_1$.
Then the inclusion $\FolLpDiff(\SFoliation,\Cman) \subset \FolLpDiff(\SFoliation,\Klein)$ is a homotopy equivalence.
The proof is similar to the formula~\eqref{equ:homotopy:incl_FixCollar__FixSandw} in the proof of the case~\ref{enum:incl_FixCollar__FixSandw} of Theorem~\ref{th:homtype:Kfol}.

On the other hand, $\FolLpDiff(\SFoliation,\Cman)$ is homeomorphic to the product $\FolLpDiff(\SFoliation,\RBd{\bConst})\times\FolLpDiff(\SFoliation,\RBd{\bConst})$ of two copies of $\FolLpDiff(\SFoliation,\RBd{\bConst})$ being weakly contractible by Remark~\ref{rem:DlpC08:contractible}.
Hence $\FolLpDiff(\SFoliation,\Cman)$ is weakly contractible as well.
Therefore $\FolLpDiff(\SFoliation,\Klein)$ is also weakly contractible.
Theorem~\ref{th:homtype:S2nxS1} is completed.


\end{document}